\documentclass[11pt]{article}
\usepackage{amssymb,latexsym,amsmath,amsthm,amsfonts, enumerate}
\usepackage{color}
\usepackage[all]{xy}

\usepackage{hyperref}
\usepackage{tikz}
\usetikzlibrary{shapes,snakes}

\def\Hom{\mathop{\rm Hom}\nolimits}

\def\Fac{\mathop{\rm Fac}\nolimits}

\def\mod{\mathop{\rm mod}\nolimits}

\def\add{\mathop{\rm add}\nolimits}

\def\tilt{\mathop{\rm tilt}\nolimits}

\def\silt{\mathop{\rm silt}\nolimits}

\def\tilt{\mathop{\rm \tau\makebox{-}tilt}\nolimits}
\def\stilt{\mathop{\rm s\tau\makebox{-}tilt}\nolimits}

\def\H{\mathop{\rm H}\nolimits}
\def\rad{\mathop{\rm rad}\nolimits}
\def\Soc{\mathop{\rm Soc}\nolimits}

\usepackage{ulem}
\begin{document}

\newcommand{\nc}{\newcommand}



\newtheorem{theorem}{Theorem}[section]
\newtheorem{proposition}[theorem]{Proposition}
\newtheorem{lemma}[theorem]{Lemma}
\newtheorem{corollary}[theorem]{Corollary}
\newtheorem{conjecture}[theorem]{Conjecture}
\newtheorem{question}[theorem]{Question}
\newtheorem{definition}[theorem]{Definition}
\newtheorem{example}[theorem]{Example}

\newtheorem{remark}[theorem]{Remark}
\def\Pf#1{{\noindent\bf Proof}.\setcounter{equation}{0}}
\def\>#1{{ $\Rightarrow$ }\setcounter{equation}{0}}
\def\<>#1{{ $\Leftrightarrow$ }\setcounter{equation}{0}}
\def\bskip#1{{ \vskip 20pt }\setcounter{equation}{0}}
\def\sskip#1{{ \vskip 5pt }\setcounter{equation}{0}}
\def\bg#1{\begin{#1}\setcounter{equation}{0}}
\def\ed#1{\end{#1}\setcounter{equation}{0}}
\def\KET{T^{^F\bot}\setcounter{equation}{0}}
\def\KEC{C^{\bot}\setcounter{equation}{0}}

\renewcommand{\thefootnote}{\fnsymbol{footnote}}
\setcounter{footnote}{0}
%
%


\title{\bf A note on the Hasse quiver of $\tau$-tilting modules
\thanks{This work was partially supported by NSFC (Grant No. 11971225). } }
\footnotetext{
E-mail:~hpgao07@163.com}
\smallskip
\author{\small Hanpeng Gao\\
{\it \footnotesize Department of Mathematics, Nanjing University, Nanjing 210093, Jiangsu Province, P.R. China}}
\date{}
\maketitle
\baselineskip 15pt
%
%
\begin{abstract}
Let $\Lambda$ be an algebra with a  indecomposable projective-injective module. Adachi gave a method to construct the Hasse quiver of support $\tau$-tilting $\Lambda$-modules. In this paper, we will show that it can be restricted to $\tau$-tilting modules.
\vspace{10pt}

\noindent {\it 2020 Mathematics Subject Classification}: 16G20, 16G60.


\noindent {\it Key words}:  $\tau$-tilting modules, Hasse quiver, support $\tau$-tilting quiver.

\end{abstract}
%
\vskip 30pt

\section{Introduction}
In 2014,   Adachi, Iyama and Reiten\cite{AIR} introduced  the concept of support $\tau$-tilting modules as a generalization of tilting modules. They showed that, in contrast to tilting modules, it is always possible to exchange a given indecomposable summand of a support $\tau$-tilting module for a unique other indecomposable and obtain a new support $\tau$-tilting module. This process, called mutation, is essential in cluster theory.  In the same paper, the authors also showed that the support $\tau$-tilting modules are in bijection with several other important classes in representation theory including
functorially finite torsion classes introduced in
 \cite{AS1981}, 
 2-term silting complexes introduced in \cite{KV1988}, and cluster-tilting objects in the cluster category.

 Let $\Lambda$ be a finite dimensional $K$-algebra   over an algebraically closed field $K$.  A  $\Lambda$-module $M$ is called {\it  $\tau$-tilting} if $\Hom_\Lambda(M,\tau M)=0$ and $|M|=|\Lambda|$. A module is called $support$ $\tau$-$tilting$ if it is a $\tau$-tilting $\Lambda/\Lambda e\Lambda$-module for some idempotent $e$ of $\Lambda$.  We will denote by $\tilt \Lambda$ (respectively, $\stilt \Lambda$) the set of isomorphism classes of basic $\tau$-tilting (respectively, support $\tau$-tilting) $\Lambda$-modules. 
For any two support $\tau$-tilting $\Lambda$-modules $M$, $N$, we write $M\geqslant N$  if $\Fac(M)\supseteq \Fac(N)$. Then $\geqslant$ gives a partial order on support $\tau$-tilting $\Lambda$-modules. The associated Hasse quiver (support $\tau$-tilting quiver) $\H(\Lambda)$ is as follows:

$\bullet$ The set of vertices is all basic support $\tau$-tilting $\Lambda$-modules.

$\bullet$ Draw an arrow from $M$ to $N$ if $M>N$ and there is no support $\tau$-tilting $\Lambda$-module $L$ such that $M>L>N$.

 Let $\mathcal{N}$ be a subposet of $\H(\Lambda)$ and $\mathcal{N}':=\H(\Lambda)\setminus \mathcal{N}$.  Adachi define a new quiver $\H(\Lambda)^{\mathcal{N}}$  from $\H(\Lambda)$ as follows(see \cite[Definition 3.2]{Adachi2016}).

$\bullet$ vertices :   vertices in $\H(\Lambda)$ and $\mathcal{N}^+$ where $\mathcal{N}^+$ is a copy of $\mathcal{N}$.

$\bullet$ arrows:  $\{a_1\to a_2\mid a_1\to a_2 \in \mathcal{N}'\} \coprod \{n_2\to a_2\mid n_2\to a_2,  n_2\in \mathcal{N}, a_2\in \mathcal{N}'\}$

~~~~~~~~\qquad$\coprod \{n_1\to n_2, n^+_1\to n^+_2\mid n_1\to n_2\in \mathcal{N}\}$$\coprod \{a_1\to n^+_1\mid a_1\to n_1, n_1\in \mathcal{N}, a_1\in \mathcal{N}'\}$

~~~~~~~~\qquad$\coprod \{n^+_1\to n_1\mid n_1\in \mathcal{N}\}.$

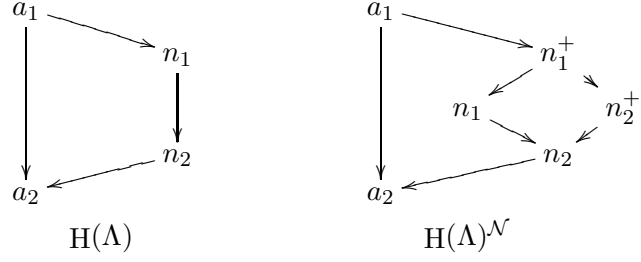
$$\xymatrix@R=2PT@C=5PT{a_1\ar[dddd]\ar[rrd]&&&&&&&&a_1\ar[dddd]\ar[rrd]&&&\\
&&n_1\ar[dd]&&&&&&&&n_1^+\ar[rd]\ar[ld]&\\
&&&&&&&&&n_1\ar[rd]&&n_2^+\ar[ld]\\
&&n_2\ar[lld]&&&&&&&&n_2\ar[lld]&\\
a_2&&&&&&&&a_2&&&\\
&\H(\Lambda)&&&&&&&&\H(\Lambda)^{\mathcal{N}}\\}.$$

Suppose that $\Lambda$ is  a basic finite dimensional $K$-algebra  with an indecomposable projective-injective module $Q$, Adachi shown that $\Soc(Q)$ is  a two-sided ideal of $\Lambda$(see \cite[Proposition 3.1(1)]{Adachi2016}).  Let $\overline{\Lambda}:=\Lambda/\Soc(Q)$ and 
$$\mathcal{N}_1:=\{N\in \stilt \overline{\Lambda}\mid Q/\Soc(Q)\in \add N~\text{and} ~\Hom_\Lambda(N,Q)=0 \}.$$
It is shown that there is an isomorphism of posets $\H(\Lambda)\to\H{(\overline{\Lambda})}^{\mathcal{N}_1}$\cite[Theorem 3.3]{Adachi2016}.

In this paper, we will show that this result can be restricted to $\tau$-tilting modules. More precisely, for an algebra $\Lambda$, let $Q(X)$ be the full subquiver of $\H( \Lambda)$ consisting of those vertices in $X$ for a sbuset $X$ of $\H(\Lambda)$. 
Now, considering the set $$\mathcal{N}:=\{N\in \stilt \overline{\Lambda}\mid Q/\Soc(Q)\in \add N,~\Hom_\Lambda(N,Q)=0~\text{and}~ |N|=|\overline \Lambda|-1 \},$$
we have the following result.
\begin{theorem}\label{1} Let $\Lambda$ be  an algebra with an indecomposable projective-injective module $Q$.
\begin{enumerate}
\item[(1)] If $Q$ is simple, then  there is an isomorphism  $$Q(\tilt \Lambda) \to Q(\tilt\overline \Lambda).$$ 
\item[(2)]  If $Q$ is not simple, then  there is an isomorphism  $$Q(\tilt \Lambda) \to Q(\tilt\overline \Lambda \coprod \mathcal{N}).$$ Moreover, if $Q/\Soc(Q)$ has $\Soc(Q)$ as a composition factor, then $\mathcal{N}=\emptyset$. Hence, we have  an isomorphism  $Q(\tilt \Lambda) \to Q(\tilt\overline \Lambda)$.
\end{enumerate}
\end{theorem}

As an application, we can calculate  the number of  $\tau$-tilting modules  over linearly Dynkin type algebras whose
square radical are zero.

Throughout this paper, all algebras will be basic,  connected, finite dimensional $K$-algebras over an algebraically closed field $K$.  For an algebra $\Lambda$, we denote by $\mod \Lambda$ 
 the category of finitely generated right $\Lambda$-modules and by $\tau$ the Auslander-Reiten
  translation of $\Lambda$.  For  $M\in \mod \Lambda$, we also  denote by $|M|$ the number of pairwise nonisomorphic indecomposable summands of $M$ and by $\add M$  the full subcategory of $\mod \Lambda$ consisting of direct summands of finite direct sums of copies of $M$. For a set $X$, we denote by $|X|$ the cardinality of $X$.  For two sets $X,Y$, $X\coprod Y$ means the disjoint union.

\section{Main results}

Let $\Lambda$ be an algebra. We always assume that $\Lambda$ has an indecomposable projective-injective module $Q$ and $\overline \Lambda:=\Lambda/\Soc (Q)$. Considering the following functor
$$\overline{(-)}:=-\otimes_\Lambda\overline \Lambda ~:\mod \Lambda\to \mod\overline \Lambda .$$
Then we have $\overline{Q}=Q/\Soc (Q)$ and $\overline{M}\cong M$ for all indecomposable $\Lambda$-modules $M$ which are not isomorphism to $Q$ by \cite[Proposition 3.1(2)]{Adachi2016}.  We will denote by  $\alpha(M)$ a basic $\Lambda$-module such that $\add(\alpha(M))=\add\overline M$.

We need the following lemma.
\begin{lemma} \label{2.1}Assume $Q$ is not simple and $U\in\mod\overline \Lambda$ does not have $\overline Q$ as a direct summand. Then
\begin{enumerate}
\item[(1)]  $U\in\tilt \Lambda$ if and only if  $U\in\tilt \overline \Lambda$.
\item[(2)] $Q\oplus\overline Q\oplus U\in\tilt \Lambda$ if and only if  $\overline Q\oplus U\in\stilt  \Lambda$ and $|\overline Q\oplus U|=|\Lambda|-1$ if and only if  $\overline Q\oplus U\in\stilt\overline  \Lambda$, $|\overline Q\oplus U|=|\overline \Lambda|-1$ and $\Hom_\Lambda(\overline Q\oplus U, Q)=0$.
\item[(3)] $Q\oplus U\in\tilt \Lambda$ if and only if  $\overline Q\oplus U\in\tilt \Lambda$ and $\Hom_\Lambda(\overline Q\oplus U, Q)\ne 0$.
\end{enumerate}
\end{lemma}
\begin{proof}
Note that a support $\tau$-tilting $\Lambda$-module $M$ is $\tau$-tilting if and only if $|M|=|\Lambda|$. Hence (1), (2) and (3) follow from \cite[Proposition 3.7]{Adachi2016}.
\end{proof}

\begin{lemma} \label{2.2} Assume $Q$ is not simple. We have
$$\{M\in\tilt \Lambda\mid Q\notin\add M,\overline Q\in \add M\}=\emptyset.$$
\end{lemma}
\begin{proof} Let $M\in \{M\in\tilt \Lambda\mid Q\notin\add M,\overline Q\in \add M\}$. Write $M=\overline Q\oplus X$ where $ X$ does not have $Q\oplus\overline Q$ as a direct summand, we have $Q\notin\Fac X$ since $Q$ is projective. By \cite[Proposition 3.7]{Adachi2016}, $Q\oplus\overline Q\oplus X\in\stilt \Lambda$ and  $|Q\oplus\overline Q\oplus X|=|\Lambda|+1$ which implies $Q\in\Fac X$. This is a contradiction.
\end{proof}

Now, we decompose $\tilt \Lambda$ as the following three parts.
$$\mathcal{M}_1:=\{M\in\tilt \Lambda\mid Q,\overline Q\notin \add M\},$$
$$\mathcal{M}_2:=\{M\in\tilt \Lambda\mid Q,\overline Q\in \add M\},$$
$$\mathcal{M}_3:=\{M\in\tilt \Lambda\mid Q\in\add M,\overline Q\notin \add M\}.$$

If $X\in\tilt\overline \Lambda$, then $X$ is sincere by \cite[Propositiin 2.2(a)]{AIR}. Hence $X$ is sincere $\Lambda$-module. Thus  $\Hom_\Lambda(N,Q)\ne 0$ since $Q$ is indecomposable injective. Therefore, $\{N\in\tilt\overline \Lambda\mid \overline Q\in\add N, \Hom_\Lambda(N,Q)= 0\}=\emptyset$. The following proposition can be obtained by Lemma \ref{2.1} immediately.
\begin{proposition} Assume that $Q$ is not simple. Then there are bijections
$$\mathcal{M}_1\to\mathcal{N}_1,~~\mathcal{M}_2\to\mathcal{N}_2,~~\mathcal{M}_3\to\mathcal{N}_3$$ given by $M\to \alpha (M)$ where
$$\mathcal{N}_1:=\{N\in\tilt\overline \Lambda\mid \overline Q\notin\add N\},$$
$$\mathcal{N}_2:=\mathcal{N}=\{N\in\stilt\overline \Lambda\mid \overline Q\in\add N, \Hom_\Lambda(N,Q)=0~\text{and}~|N|=|\overline \Lambda|-1\},$$
$$\mathcal{N}_3=\{N\in\tilt\overline \Lambda\mid \overline Q\in\add N, \Hom_\Lambda(N,Q)\ne 0\}.$$
In particular, there is a bijection 
$$\alpha: \tilt \Lambda\to \tilt\overline \Lambda \coprod \mathcal{N}.$$
\end{proposition}

\begin{corollary} We have
$$\tilt \Lambda=\{N\mid N\in \mathcal{N}_1\}\coprod\{Q\oplus N\mid N\in\mathcal{N}_2\}\coprod \{Q\oplus(N/\overline Q)\mid N\in \mathcal{N}_3\}.$$

\end{corollary}

Now, we are ready to prove Theorem \ref{1}.

{\bf Proof of Theorem \ref{1}}. (1) It is clearly since $\tilt \Lambda=\{Q\oplus M\mid M\in\tilt \overline \Lambda\}$ where $Q$ is a simple projective-injective $\Lambda$-module.

(2)  By Proposition \ref{2.2}, we have a bijection$$\alpha: \tilt \Lambda\to \tilt\overline \Lambda \coprod \mathcal{N}.$$
We only need to show that, for any $M,L\in\tilt \Lambda$, $M\geqslant L$ if and only if $\alpha(M)\geqslant \alpha(L)$. In fact, if  $M\geqslant L$, then $L\in\Fac M$ and we have $\overline L\in \Fac \overline M$ which implies $\alpha(M)\geqslant \alpha(L)$.
Conversely, let $\alpha(M)\geqslant \alpha(L)$. If both $M$ and $L$ are in  $\mathcal{M}_i (i=1,2,3)$, then it is clear that $M\geqslant L$.  Otherwise,

{\bf Case 1}: If $L\in \mathcal{M}_1$, then $\alpha(L)=L$. Note that $\overline M\in \Fac M$, we have $M\geqslant \alpha(M)\geqslant \alpha(L)=L$.

{\bf Case 2}: If $L\in \mathcal{M}_2$, then $M\in  \mathcal{M}_3$ since  $\mathcal{N}_1$ has no $\overline Q$ as a direct summand. Thus $\alpha(L)=Q\oplus \overline L\in \Fac M$ because $\mathcal{M}_3$ has  $Q$ as a direct summand.

{\bf Case 3}: If $L\in \mathcal{M}_3$, then $M\notin  \mathcal{M}_1$ since  $\mathcal{N}_1$ has no $\overline Q$ as a direct summand. Assume  $M\in  \mathcal{M}_2$. Then $\alpha(M)$ has no $\Soc (Q)$ as a composition factor and $\alpha(L)$ has  $\Soc (Q)$ as a composition factor. This is a contradiction with $\alpha(L)\in \Fac\alpha(M)$.

Thus the assertion follows.

We illustrate Theorem \ref{1} with the following example.

\begin{example}\label{4.1}
{\rm Let $\Lambda$ be a finite dimensional $K$-algebra given by the quiver
$1 \stackrel{\alpha}{\longrightarrow} 2 \stackrel{\beta}{\longrightarrow} 3$ with the relation $\alpha\beta=0$. Take $Q=P_1$. Then $Q$ is an indecomposable projective-injective module. The algebra $\overline \Lambda$ given by the quiver
$1 ~~~~~~ 2 \stackrel{\beta}{\longrightarrow} 3.$

We draw the Hasse quivers $\H(\Lambda)$ and $\H( \overline \Lambda)$ as follows, 
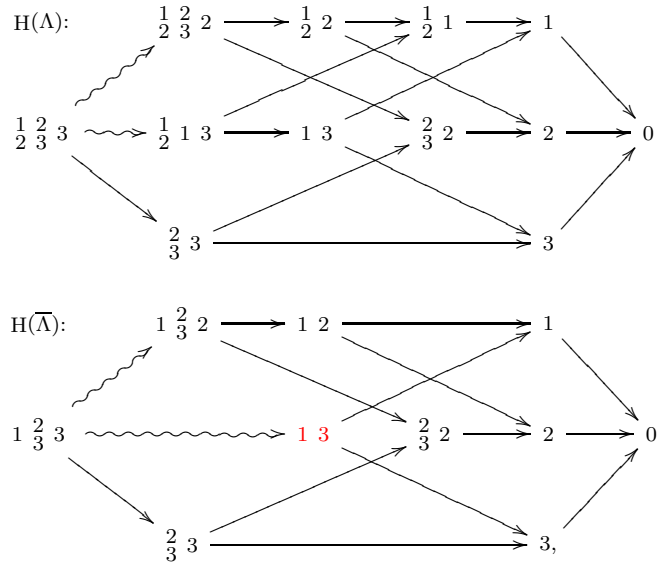
\[\xymatrix{ {\smallmatrix \H(\Lambda):
\endsmallmatrix} &{\smallmatrix 1\\2
\endsmallmatrix}{\smallmatrix 2\\3
\endsmallmatrix}{\smallmatrix 2
\endsmallmatrix}\ar[r]\ar[rrd]&{\smallmatrix 1\\2
\endsmallmatrix}{\smallmatrix 2
\endsmallmatrix}\ar[r]\ar[rrd]&{\smallmatrix 1\\2
\endsmallmatrix}{\smallmatrix 1
\endsmallmatrix}\ar[r]&{\smallmatrix 1
\endsmallmatrix}\ar[rd]&\\
{\smallmatrix 1\\2
\endsmallmatrix}{\smallmatrix 2\\3
\endsmallmatrix}{\smallmatrix 3
\endsmallmatrix}\ar@{~>}[r]\ar@{~>}[ur]\ar[rd]&{\smallmatrix 1\\2
\endsmallmatrix}{\smallmatrix 1
\endsmallmatrix}{\smallmatrix 3
\endsmallmatrix}\ar[r]\ar[rru]&{\smallmatrix 1
\endsmallmatrix}{\smallmatrix 3
\endsmallmatrix}\ar[rru]\ar[rrd]&{\smallmatrix 2\\3
\endsmallmatrix}{\smallmatrix 2
\endsmallmatrix}\ar[r]&{\smallmatrix 2
\endsmallmatrix}\ar[r]&{\smallmatrix 0
\endsmallmatrix}\\
&{\smallmatrix 2\\3
\endsmallmatrix}{\smallmatrix 3
\endsmallmatrix}\ar[rrr]\ar[rru]&&&{\smallmatrix 3
\endsmallmatrix}\ar[ru]&
}\]

\[\xymatrix{ {\smallmatrix \H(\overline\Lambda):
\endsmallmatrix} &{\smallmatrix 1
\endsmallmatrix}{\smallmatrix 2\\3
\endsmallmatrix}{\smallmatrix 2
\endsmallmatrix}\ar[r]\ar[rrd]&{\smallmatrix 1
\endsmallmatrix}{\smallmatrix 2
\endsmallmatrix}\ar[rr]\ar[rrd]&&{\smallmatrix 1
\endsmallmatrix}\ar[rd]&\\
{\smallmatrix 1
\endsmallmatrix}{\smallmatrix 2\\3
\endsmallmatrix}{\smallmatrix 3
\endsmallmatrix}\ar@{~>}[rr]\ar@{~>}[ur]\ar[rd]&&{\smallmatrix \color{red}{1}
\endsmallmatrix}{\smallmatrix \color{red}{3}
\endsmallmatrix}\ar[rru]\ar[rrd]&{\smallmatrix 2\\3
\endsmallmatrix}{\smallmatrix 2
\endsmallmatrix}\ar[r]&{\smallmatrix 2
\endsmallmatrix}\ar[r]&{\smallmatrix 0
\endsmallmatrix}\\
&{\smallmatrix 2\\3
\endsmallmatrix}{\smallmatrix 3
\endsmallmatrix}\ar[rrr]\ar[rru]&&&{\smallmatrix 3,
\endsmallmatrix}\ar[ru]&
}\]

We draw those arrows in  $Q(\tilt \Lambda)$ and $Q(\tilt \overline \Lambda\coprod \mathcal{N})$
by\xymatrix{\ar@{~>}[r] &} and  $\mathcal{N}$ is marked by red.}
\end{example}

Considering  the following quivers.
 $$\xymatrix@!@R=5pt@C=5pt{
A_n:&n\ar[r]&n-1\ar[r]&\cdots\ar[r]&2\ar[r]&1}$$

$$\xymatrix@R=5PT@C=15PT{&&&&1\\
D_n:~~~n\ar[r]&n-1\ar[r]&\cdots\ar[r]&3\ar[ru]\ar[rd]&\\
&&&&2}$$

Take $A_n^2:=KA_n/\rad^2$  and $D_n^2:=KD_n/\rad^2$.  Applying our results, we can give   a recurrence relation about the numbers of $\tau$-tilting modules over $A_n^2$ and $D_n^2$.

\begin{theorem} Let $\Lambda^2_n$  be an algebra ($A_n^2$ or $D_n^2$). Then we have 
 $$|\tilt \Lambda^2_n|=|\tilt \Lambda^2_{n-1}|+|\tilt \Lambda^2_{n-2}|.$$
\end{theorem}
\begin{proof} Take  $Q=P_n$ which is an indecomposable projective-injective $\Lambda^2_n$-module. Since $\Soc (Q)\cong S_{n-1}$, we have $\overline{\Lambda^2_n}=\Lambda^2_n/S_{n-1}\cong \Lambda^2_{n-1}\times K$ and $Q/S_{n-1}\cong S_n$. Hence
\begin{equation*}
\begin{split}
\mathcal{N}&=\{N\in \stilt \overline{\Lambda^2_n}\mid Q/\Soc(Q)\in \add N,\Hom_{\Lambda^2_n}(N,Q)=0 ~\text{and} ~|N|=n-1\}\\
&=\{N\in \stilt (\Lambda^2_{n-1}\times K)\mid S_n\in \add N,\Hom_{\Lambda^2_n}(N,P_n)=0 ~\text{and} ~|N|=n-1 \}\\
&=\{S_n\oplus L\mid L\in \silt \Lambda^2_{n-1}, \Hom_{\Lambda^2_n}(L,P_n)=0 ~\text{and} ~|L|=n-2 \}\\
&=\{S_n\oplus L\mid L\in \silt \Lambda^2_{n-1},\Hom_{\Lambda^2_{n-1}}(L,S_{n-1})=0 ~\text{and} ~|L|=n-2 \}\\
&=\{S_n\oplus L\mid L\in \silt \Lambda^2_{n-2}~\text{and} ~|L|=n-2 \}\\
&=\{S_n\oplus L\mid L\in \tilt \Lambda^2_{n-2} \}.
\end{split}
\end{equation*}
By Theorem \ref{1}, there is a bijection $Q(\tilt \Lambda^2_n) \to Q(\tilt({\Lambda^2_{n-1}\times K}) \coprod \mathcal{N})$. Thus
$$|\tilt \Lambda^2_n|=|\tilt (\Lambda^2_{n-1}\times K)|+| \mathcal{N}|=|\tilt \Lambda^2_{n-1}|+|\tilt \Lambda^2_{n-2}|.$$
\end{proof}

\begin{corollary}~
\begin{enumerate}
\item[(1)]  $|\tilt A^2_n|=\frac{(1+\sqrt{5})^{n+1}-(1-\sqrt{5})^{n+1}}{\sqrt{5}\cdot 2^{n+1}}.$
\item[(2)]  $|\tilt D^2_n|=\frac{(2\sqrt{5}-1)(1+\sqrt{5})^{n-1}+(2\sqrt{5}+1)(1-\sqrt{5})^{n-1}}{\sqrt{5}\cdot 2^{n-1}}.$
\end{enumerate}
\end{corollary}

\subsection*{Acknowledgements}

The author would like to thank Professor Zhaoyong Huang for helpful discussions. He also  thanks  the referee for the useful and detailed suggestions.  This work was partially supported by the National natural Science Foundation of China (No. 11971225).

\end{document}